\documentclass[11pt]{amsart}
\usepackage{
 amsmath, 
 amsxtra, 
 amsthm, 
 amssymb, 
 etex, 
 mathrsfs, 
 mathtools, 
 tikz-cd, 
 bbm,
 xr,
 comment,
 enumitem}
\usepackage[toc]{appendix} 
\usepackage[all]{xy}
\usepackage{hyperref}
\usepackage{url}
\usepackage{tipa}
\usepackage{dsfont}
\usepackage{ textcomp }
\usepackage{stmaryrd} 
\usepackage{amssymb}
\usepackage{mathabx} 
\usepackage{amsmath} 
\usepackage{amscd}
\usepackage{amsbsy}
\usepackage{comment, enumerate}
\usepackage{color}
\usepackage{mathtools,caption}
\usepackage{tikz-cd}
\usepackage{longtable}
\usepackage[utf8]{inputenc}
\usepackage[OT2,T1]{fontenc}
\usepackage{changepage}
\usepackage{commath}

\DeclareMathOperator{\Hom}{Hom}

\DeclareMathOperator{\Gal}{Gal}
\DeclareMathOperator{\rank}{rank}

\DeclareMathOperator{\Sel}{Sel}
\DeclareMathOperator{\coker}{coker}

\DeclareMathOperator{\cyc}{cyc}

\newtheorem{theorem}{Theorem}[section]
\newtheorem*{theorem*}{Theorem}
\newtheorem{lemma}[theorem]{Lemma}

\newtheorem{proposition}[theorem]{Proposition}
\newtheorem{corollary}[theorem]{Corollary}
\newtheorem{defn}[theorem]{Definition}

\numberwithin{equation}{section}
\newtheorem{lthm}{Theorem} 

\theoremstyle{remark}
\newtheorem{remark}[theorem]{Remark}

\newcommand\EatDot[1]{}

\newcommand{\Ep}{E[p^\infty]}

\newcommand{\?}{\stackrel{?}{=}}

\newcommand{\Qcycp}{\QQ_{\cyc,p}}

\newcommand{\cM}{{\mathcal{M}}}
\newcommand{\Image}{\mathrm{Image}}
\newcommand{\loc}{\mathrm{loc}}

\newcommand{\QQ}{\mathbb{Q}}
\newcommand{\ZZ}{\mathbb{Z}}
\newcommand{\Qp}{\mathbb{Q}_p}
\newcommand{\Zp}{\mathbb{Z}_p}

\definecolor{Green}{rgb}{0.0, 0.5, 0.0}

\newcommand{\Qcyc}{\QQ_{\cyc}}

\newcommand{\Char}{\mathrm{Char}}
\newcommand{\MW}{\mathrm{MW}}

  \DeclareFontFamily{U}{wncy}{}
  \DeclareFontShape{U}{wncy}{m}{n}{<->wncyr10}{}
  \DeclareSymbolFont{mcy}{U}{wncy}{m}{n}
  \DeclareMathSymbol{\sha}{\mathord}{mcy}{"58}
  \DeclareMathSymbol{\zhe}{\mathord}{mcy}{"11}
\title[Fine/plus/minus Mordell--Weil groups]{Algebraic structure and characteristic ideals of fine Mordell--Weil groups and plus/minus Mordell--Weil groups}

\makeatletter
\let\@wraptoccontribs\wraptoccontribs
\makeatother

\author[A.~Lei]{Antonio Lei}
\address{Antonio Lei\newline Department of Mathematics and Statistics\\University of Ottawa\\
150 Louis-Pasteur Pvt\\
Ottawa, ON\\
Canada K1N 6N5}
\email{antonio.lei@uottawa.ca}


\begin{document}
\begin{abstract}
Given an elliptic curve defined over a number field $F$, we study the algebraic structure and prove a control theorem for Wuthrich's fine Mordell--Weil groups over a $\mathbb{Z}_p$-extension of $F$, generalizing results of Lee on the usual Mordell--Weil groups. In the case where $F=\mathbb{Q}$, we show that the characteristic ideal of the Pontryagin dual of the fine Mordell--Weil group over the cyclotomic $\mathbb{Z}_p$-extension coincides with Greenberg's prediction for the characteristic ideal of the dual fine Selmer group. If furthermore $E$ has good supersingular reduction at $p$ with $a_p(E)=0$, we generalize Wuthrich's fine Mordell--Weil groups to define "plus and minus Mordell--Weil groups". We show that the greatest common divisor of the characteristic ideals of the Pontryagin duals of these groups coincides with Kurihara--Pollack's prediction for the greatest common divisor of the plus and minus $p$-adic $L$-functions.
\end{abstract}

\maketitle

\section{Introduction}
\label{S: Intro}
Throughout this article, $p$ is a fixed odd prime number. Let $F$ be a number field and $F_\infty/F$  a $\Zp$-extension with Galois group $\Gamma$. We write $\Lambda$ for the Iwasawa algebra $\Zp[[\Gamma]]$, which we identify with the ring of power series $\Zp[[X]]$ via $X=\gamma-1$, where $\gamma$ is a topological generator of $\Gamma$. Given an integer $n\ge1$, we write $\Phi_n=\displaystyle\frac{(X+1)^{p^n}-1}{(X+1)^{p^{n-1}}-1}\in \Lambda$ for the $p^n$-th cyclotomic polynomial in $X+1$. For $n=0$, we define $\Phi_0=X$.

Let $E/F$ be an elliptic curve. In \cite{lee20}, Lee studied the structure of the $\Lambda$-module $E(F_\infty)\otimes_\ZZ\Qp/\Zp$ and showed that there is a pseudo-isomorphism of $\Lambda$-modules
\begin{equation}
\label{eq:pseudo-Lee}
    \left(E(F_\infty)\otimes_\ZZ\Qp/\Zp\right)^\vee\sim\Lambda^r\oplus\bigoplus_{i=1}^t\Lambda/\Phi_{b_i}
\end{equation}
for certain non-negative integers $r,t$ and $b_i$, and $M^\vee$ denotes the Pontryagin dual of $M$ (see Theorem A in op. cit.). Let $F_n$ denote the unique sub-extension of $F_\infty/F$ such that $[F_n:F]=p^n$. Consider the natural map
\[
\MW_n:E(F_n)\otimes_\ZZ\Qp/\Zp\rightarrow \left(E(F_\infty)\otimes_\ZZ\Qp/\Zp\right)^{\Gamma_n},
\]
where $\Gamma_n=\Gal(F_\infty/F_n)$.
Lee  showed that $\ker(\MW_n)$ is finite, of order bounded independently of $n$ (Lemma~2.0.1 of op. cit.) and that if $r=0$, then $\MW_n$ is surjective for almost all $n$ (Theorem~2.1.5 of op. cit.). This in turn tells us that the Mordell--Weil ranks of $E(F_n)$ stabilizes as $n\rightarrow \infty$. If furthermore the $p$-primary Selmer groups $\Sel_{p^\infty}(E/F_n)$ satisfy a control theorem (as proved by Mazur \cite{mazur72} in the case where $E$ has good ordinary reduction at all the primes above $p$), the factors $\Lambda/\Phi_{b_i}$ can be described by the growth of Mordell--Weil ranks of $E(F_n)$ as $n$ grows (see \cite[Remark~2.1.6(2)]{lee20}).

The starting point of the present article is a generalization of Lee's results to the setting of fine Mordell--Weil groups introduced by Wuthrich in \cite{wut-camb}. These groups are important in the study of the fine Selmer group $\Sel_0(E/F_\infty)$ of $E$ over $F_\infty$, as defined by
Coates and Sujatha \cite{CoatesSujatha_fineSelmer}. It is conjectured that when $F_\infty/F$ is the cyclotomic $\Zp$-extension, the Pontryagin dual $\Sel_0(E/F_\infty)^\vee$ is a finitely generated $\Zp$-module (labeled Conjecture A in op. cit.). One advantage of studying the Iwasawa theory of the fine Selmer group over the classical $p$-primary Selmer group $\Sel_{p^\infty}(E/F_\infty)$ is that it behaves more closely to the $p$-primary part of class groups studied in classical Iwasawa Theory. Furthermore, many Iwasawa-theoretic results on fine Selmer groups can be stated in a uniform way, independent of the reduction type of $E$ at $p$.

Let $\cM(E/K)$ denote Wuthrich's  fine Mordell--Weil group of $E$ over an extension $K$ of $F$  (see Definition~\ref{defn:fine}). Similar to the classical $p$-primary Selmer group, which sits inside the short exact sequence
\[
0\rightarrow E(K)\otimes\Qp/\Zp\rightarrow \Sel_{p^\infty}(E/K)\rightarrow \sha(E/K)[p^\infty]\rightarrow0,
\]
the fine Selmer group sits inside the short exact sequence
\begin{equation}\label{eq:fineSel}
   0\rightarrow \cM(E/K)\rightarrow \Sel_0(E/K)\rightarrow \zhe(E/K)\rightarrow0, 
\end{equation}
where $\zhe(E/K)$ is isomorphic to a subgroup of $\sha(E/K)[p^\infty]$ and is referred to as the fine Tate--Shafarevich group by Wuthrich. One interesting application is that this allows us to study Coates--Sujatha's Conjecture A by analyzing the growth of $|\zhe(E/F_n)|$, which is conjecturally finite, as $n\rightarrow \infty$ (see \cite[Conjecture~8.2]{wut-camb}).

In this paper, we study two generalizations of Lee's results in \cite{lee20} to Wuthrich's fine Mordell--Weil groups.
Our first result studies the algebraic structure of fine Mordell--Weil groups over $F_\infty$:
\begin{lthm}[Theorem~\ref{thm:structure}]\label{thmA}
If the integer $r$ in \eqref{eq:pseudo-Lee} is zero, then there is a pseudo-isomorphism of $\Lambda$-modules
\[
\cM(E/F_\infty)^\vee\sim \bigoplus_{i=1}^u \Lambda/\Phi_{c_i}
\]
for certain non-negative integers $u$ and $c_i$. 
\end{lthm}

Our second result  is the following control theorem for the fine Mordell--Weil groups:

\begin{lthm}[Theorem~\ref{thm:control}]\label{thmB}
Let $m_n$ denote the natural morphism
\[
m_n:\cM(E/F_n)\rightarrow \cM(E/F_\infty)^{\Gamma_n}
\]
induced by the inclusion $E(F_n)\rightarrow E(F_\infty)$.
\begin{itemize}
    \item[(a)]  The kernel of $m_n$ is finite of order bounded independently of $n$. 
    \item[(b)]Suppose that $\zhe(E/F_n)$ is finite, then the  cokernel of $m_n$ is finite. 
  \end{itemize}
\end{lthm}

Note that unlike the case of Mordell--Weil groups studied in \cite{lee20}, our control theorem is independent of the reduction type of $E$. Our proof relies on the corresponding control theorem for the fine Selmer groups  proved by Lim \cite{Lim-control}.

Armed with these two general results, we specialize to the case where $F=\QQ$ and $F_\infty=\Qcyc$ is the cyclotomic $\Zp$-extension of $\QQ$.  In particular, we study  the characteristic ideal of the Pontryagin dual of the fine Mordell--Weil group over $\Qcyc$, which is equivalent to studying the direct summands appearing on the right-hand side of the pseudo-isomorphism in Theorem~\ref{thmA}. For a positive integer $n$, we write $e_n=\displaystyle\frac{\rank E(F_n)-E(F_{n-1})}{\phi(p^n)}$, where $\phi$ is the Euler totient function. We also define $e_0=\rank E(\QQ)$.
 Greenberg posed the following problem
 \[
        \Char_\Lambda\Sel_0(E/\Qcyc)^\vee\?\left(\prod_{e_n>0}\Phi_n^{e_n-1}\right)
 \]
  (see \cite[Problem~0.7]{KP}). 
The following theorem tells us that the right-hand side can in fact be reinterpreted as the characteristic ideal of  $\cM(E/\Qcyc)^\vee$.
\begin{lthm}[Theorem~\ref{thm:char-fineMW}]\label{thmC}
Suppose that $\sha(E/F_n)[p^\infty]$ is finite for all $n\ge0$. Then,
\[
\Char_\Lambda \cM(E/\Qcyc)^\vee=\left(\prod_{e_n>0}\Phi_n^{e_n-1}\right).
\]
\end{lthm}
In particular, thanks to \eqref{eq:fineSel}, this tells us that under the hypothesis of Theorem~\ref{thmC}, the answer to Greenberg's problem is affirmative if and only if $\zhe(E/\Qcyc)$ is finite. Recall that $\zhe(E/\Qcyc)$ is isomorphic to a subgroup of $\sha(E/\Qcyc)[p^\infty]$ and its definition depends on the chosen prime $p$. It is predicted by Wuthrich that $\zhe(E/\Qcyc)$ should be finite for all but finitely many $p$ (see \cite[Conjecture~8.4]{wut-camb}).

In the case where $E/\QQ$ has good supersingular reduction at $p$  with $a_p(E)=0$, Kurihara and Pollack have posed the following problem:
\[
\gcd\left(L_p^+,L_p^-\right)\?\left(X^{e_0}\prod_{n>0,e_n>0}\Phi_n^{e_n-1}\right),
\]
where $L_p^\pm$ are Pollack's $p$-adic $L$-functions defined in \cite{pollack03} (see \cite[Problem~3.2]{KP}). Under certain hypotheses, relations between Greenberg's problem and Kurihara--Pollack's problem have been studied in \cite[\S3]{KP}. Under different hypotheses, further links between  $\gcd\left(L_p^+,L_p^-\right)$ and the characteristic ideal of $ \Sel_0(E/\Qcyc)^\vee$ have recently been studied in \cite{LeiSujatha2}.

Theorem~\ref{thmC} has led us to study an analogue of Kurihara--Pollack's problem in terms of subgroups of $E(\Qcyc)\otimes\Qp/\Zp$. In particular, we define what we call "plus and minus Mordell--Weil groups", denoted by $\cM^\pm(E/K)$, for sub-extensions $K$ of $\Qcyc$, utilizing the local plus and minus subgroups of Kobayashi studied in  \cite{kobayashi03} (see Definition~\ref{def:pm}). We will show that analogues of Theorems~\ref{thmA} and \ref{thmB} hold for these new groups (see Theorem~\ref{thm:structure-pm}).

Similar to the  parallel between Theorem~\ref{thmC} and Greenberg's problem, we prove that the description of $\gcd\left(L_p^+,L_p^-\right)$ predicted by Kurihara--Pollack's problem is in fact equal to the greatest common divisor of the characteristic ideals of the Pontryagin duals of these new plus and minus Mordell--Weil groups:

\begin{lthm}[Theorem~\ref{thm:Char-pm}]\label{thmD}
Suppose that $E/\QQ$ has supersingular reduction at $p$ with $a_p(E)=0$ and that $\sha(E/F_n)[p^\infty]$ is finite for all $n\ge0$. Then,
\[
\gcd\left(\Char_\Lambda\cM^+(E/\Qcyc)^\vee,\Char_\Lambda\cM^-(E/\Qcyc)^\vee\right)=\left(X^{e_0}\prod_{n>0,e_n>0}\Phi_n^{e_n-1}\right).
\]
\end{lthm}

Similar to Theorem~\ref{thmC}'s consequence on Greenberg's problem, we shall see that Theorem~\ref{thmD} tells us that the answer to Kurihara--Pollack's problem is affirmative under the hypotheses that certain plus and minus Tate--Shafarevich groups are finite and that Kobayashi's plus and minus main conjecture  holds (see Remark~\ref{rk:KP} at the end of the article).
\section*{Acknowledgements}
The author thanks Katharina M\"uller for interesting dicussions on topics studied in this article.  He is also indebted to the anonymous referee for their careful reading of an earlier version of the article and their extremely helpful comments and suggestions. The author's research is supported by the NSERC Discovery Grants Program RGPIN-2020-04259 and RGPAS-2020-00096.

\section{Notation}
\label{sec:notation}

Throughout the article, $E$ is an elliptic curve defined over a number field $F$. We fix a $\Zp$-extension $F_\infty/F$ and let $\Gamma$, $\Gamma_n$, $\Lambda$ and $\Phi_n$ be defined as in the introduction. We introduce further basic notation that will be used in subsequent sections of the article.

\begin{defn}
Let $M$ be a $\Lambda$-module and $n\ge0$ an integer. 
\begin{itemize}
    \item We write $M^{\Gamma_n}=H^0(\Gamma_n,M)$ and $M_{\Gamma_n}=H_0(\Gamma_n,M)$.
    \item We write $M[\Phi_n]=\{m\in M:\Phi_n\cdot m=0\}$.
\end{itemize}
\end{defn}

\begin{defn}
Let $A$ be an abelian group.
\begin{itemize}
    \item The Tate module of $A$, denoted by $T_pA$, is defined as the projective limit $\varprojlim_k A[p^k]$, where the connecting maps are given by $\times p$.
   
    \item The $p$-adic completion $A^\star$ of $A$ is defined to be $\varprojlim_k A/p^kA$, where the connecting maps are projections.
     \item We define the $\Qp$-linear versions of $T_pA$ and $A^\star$ by
     \begin{align*}
         V_pA=T_pA\otimes_{\Zp} \Qp,\quad A^\bullet =A^\star\otimes_{\Zp} \Qp.
     \end{align*}
\end{itemize}
\end{defn}
\begin{defn}Let $M$ be a  $\Zp$-module. 
\begin{itemize}
    \item We write $M^\vee$ for the Pontryagin dual $\Hom(M,\Qp/\Zp)$.
    \item We write $M^*$ for the linear  dual $\Hom(M,\Zp)$.
    \item We say that $M$ is cofinitely generated over $\Zp$ if $M^\vee$ is finitely generated over $\Zp$.
    \item If $M$ is a $\Lambda$-module, we say that $M$ is cofinitely generated over $\Lambda$ if $M^\vee$ is finitely generated over $\Lambda$.
\end{itemize}
\end{defn}

\begin{remark}\label{rk:observe}
Let $M$ be a $\Lambda$-module that is cofinitely generated over $\Zp$.  There is a pseudo-isomorphism
$$M^\vee\sim \bigoplus_{i=1}^k \Lambda/f_i,$$
where $f_i$ are distinguished polynomials.
This allows us to make the following observations:
\begin{itemize}
    \item[(a)] There is an $\Lambda$-isomorphism $$T_pM\cong  \bigoplus_{i=1}^k \Lambda/f_i^\iota,$$
where $\iota$ is the $\Zp$-linear inversion $\Lambda\rightarrow \Lambda$ sending $\gamma$ to $\gamma^{-1}$.
    \item[(b)] There is a pseudo-isomorphism
    $$M^\vee\sim T_pM^*  $$
    of $\Lambda$-modules.
\end{itemize}

\end{remark}

\section{Structure theorem and Control theorem for fine Mordell--Weil groups}

We first recall the definition of fine Mordell--Weil groups of Wuthrich given in \cite{wut-camb}.
\begin{defn}\label{defn:fine}
Given an integer $k\ge1$ and an extension $K/F$, the $p^k$-fine Mordell--Weil group of $E$ over $K$ is defined by
\[
M_{p^k}(E/K)=\ker\left(E(K)/p^k\rightarrow \bigoplus_{v|p} E(K_v)/p^kE(K_v)\right).
\]
The ($p$-primary) fine Mordell--Weil group of $E$ over $K$ is defined  by
\[
\cM(E/K)=\varinjlim_k M_{p^k}(E/K)=\ker\left(E(K)\otimes\Qp/\Zp\rightarrow \bigoplus_{v|p} E(K_v)\otimes\Qp/\Zp\right).
\]
\end{defn}
Since $p$ is fixed throughout the article, we omit the term "$p$-primary" in our discussion.  We recall from the discussion in \cite[\S2]{wut-camb} that the Kummer map induces an injection $\cM(E/K)\hookrightarrow\Sel_0(E/K)$. Here $\Sel_0(E/K)$ is the fine Selmer group given by
\[\ker\left(H^1(K,\Ep)\rightarrow\prod_v\frac{H^1(K_v,\Ep)}{E(K_v)\otimes\Qp/\Zp}\right),\]
 where the product runs over all places $v$ of $K$. The aforementioned injection allows us to define:

\begin{defn}
    The fine Tate--Shafarevich group of $E$ over $K$, denoted by $\zhe(E/K)$, is defined to be the quotient $\Sel_0(E/K)/\cM(E/K)$.
\end{defn}
It particular, this gives tautologically the short exact sequence \eqref{eq:fineSel}.

We now give the  proof of Theorem~\ref{thmA} stated in the introduction, which is a simple consequence of the pseudo-isomorphism \eqref{eq:pseudo-Lee} proved by Lee.
\begin{theorem}\label{thm:structure}
If $r=0$ in \eqref{eq:pseudo-Lee}, then there is a pseudo-isomorphism of $\Lambda$-modules
\[
\cM(E/F_\infty)^\vee\sim\bigoplus_{i=1}^u\Lambda/\Phi_{c_i}
\]
for certain non-negative integers  $u$ and $c_i$.
\end{theorem}
\begin{proof}
By definition, $\cM(E/F_\infty)^\vee$ is  a quotient of $\left(E(F_\infty)\otimes \Qp/\Zp\right)^\vee$. Since the direct summands on the right-hand side of \eqref{eq:pseudo-Lee} are all simple cyclic modules when $r=0$, it follows that $\cM(E/F_\infty)^\vee$ is pseudo-isomorphic to a partial direct sum of the right-hand side of \eqref{eq:pseudo-Lee}. 
\end{proof}

\begin{remark}
It is not clear to us whether the proof of \cite[Theorem~2.1.2]{lee20} can be generalized to the setting of fine Mordell--Weil groups without assuming $r=0$. Lee's proof relies on the fact that $E(F_n)\otimes\Qp/\Zp$ is divisible, whereas $\cM(E/F_n)$ is not divisible in general (see \cite[discussion before Theorem~7.1]{wut-camb}).\end{remark}
\begin{remark}
\label{rk:Kato-Roh}
When $E$ is defined over $\QQ$, the extension $F/\QQ$ is an abelian extension,  and $F_\infty$ is the cyclotomic $\Zp$-extension of $F$, the works of Kato \cite{Kato} and Rohrlich \cite{rohrlich88} tell us that $r=0$. In particular,  Theorem~\ref{thm:structure} applies in this setting (see \cite[Theorem~1.1.6]{lee20}).
\end{remark}

We now proceed to prove a control theorem for the fine Mordell--Weil groups (Theorem~\ref{thmB}).

\begin{theorem}\label{thm:control}Let $m_n$ denote the natural morphism
\[
m_n:\cM(E/F_n)\rightarrow \cM(E/F_\infty)^{\Gamma_n}
\]
induced by the inclusion $E(F_n)\rightarrow E(F_\infty)$.
\begin{itemize}
    \item[(a)]  The kernel of $m_n$ is finite, with order bounded independently of $n$. 
    \item[(b)]Suppose that $\zhe(E/F_n)$ is finite, then the  cokernel of $m_n$ is finite.
\end{itemize}
\end{theorem}
\begin{proof}
Consider the commutative diagram:
\[
\xymatrix{
0\ar[r]&\cM(E/F_\infty)^{\Gamma_n}\ar[r]&\Sel_0(E/F_\infty)^{\Gamma_n}\ar[r]&\zhe(E/F_\infty)^{\Gamma_n}\\
0\ar[r]&\cM(E/F_n)\ar[u]^{m_n}\ar[r]&\Sel_0(E/F_n)\ar[u]^{\beta_n}\ar[r]&\zhe(E/F_n)\ar[u]^{\gamma_n}\ar[r]&0.
}
\]
By \cite[Theorem~1.1]{Lim-control}, the kernel and cokernel of $\beta_n$ are finite of bounded orders. Therefore,  the kernel of $m_n$ is  finite, with order bounded independently of $n$, which proves part (a) of the theorem. 

Under our hypothesis in (b), $\ker \gamma_n$ is finite. By the snake lemma, the finiteness of $\coker m_n$ follows and part (b) of the theorem is proved.
\end{proof}
\begin{remark}\label{rk:finite}
As discussed in \cite[P.3]{wut-camb}, $\zhe(E/F_n)$ is isomorphic to a subgroup of $\sha(E/F_n)[p^\infty]$. In particular, the finiteness of $\zhe(E/F_n)$ would follow from that of $\sha(E/F_n)$, which is expected to hold unconditionally.
\end{remark}

\begin{corollary}\label{cor:Tate-fine}
Suppose that $r=0$ in \eqref{eq:pseudo-Lee} and that $\zhe(E/F_n)$ is finite. There is a $\Lambda$-isomorphism $$T_p\cM(E/F_n)\cong\bigoplus_{c_i\le n} \Lambda/\Phi_{c_i},$$
where the integers $c_i$ are given as in Theorem~\ref{thm:structure}.
\end{corollary}
\begin{proof}
Since $\left(\Lambda/\Phi_{c_i}\right)_{\Gamma_n}=\Lambda/(\Phi_{c_i},\omega_n)$, where $\omega_n=(1+X)^{p^n}-1$, has finite cardinality for all $c_i>n$,  Theorem~\ref{thm:control} tells us that $\cM(E/F_n)^\vee$ is given by, up to finite modules, $\bigoplus_{c_i\le n}\Lambda/\Phi_{c_i}$. Furthermore, $\Lambda/\Phi_{c_i}^\iota= \Lambda/\Phi_{c_i}$ since the ideal $(\Phi_{c_i})$ is invariant under $\iota$. Hence, the asserted isomorphism follows from Remark~\ref{rk:observe}.
\end{proof}

\section{The characteristic ideal of the dual of the fine Mordell--Weil group over the cyclotomic $\Zp$-extension of $\QQ$}From now on, we take $F=\QQ$. In particular, $F_\infty$ is the cyclotomic $\Zp$-extension of $\QQ$, which we denote by $\Qcyc$. We shall write $\QQ_{(n)}$ for $F_n$. Note that $(E(\Qcyc)\otimes\Qp/\Zp)^\vee$ is a torsion $\Lambda$-module and Theorem~\ref{thm:structure} applies as discussed in Remark~\ref{rk:Kato-Roh}. 

We suppose from now on that $\sha(E/\QQ_{(n)})[p^\infty]$ is finite for all $n\ge0$. In particular, $\zhe(E/\QQ_{(n)})$ is also finite.

For all subextensions $K$ of $\Qcyc$, there is exactly one place lying above $p$. For simplicity, we shall write $K_p$ for the completion of $K$ at this unique place.

\begin{defn}
Let $n\ge0$ be an integer.
\begin{itemize}
    \item We write  $s_n$ for the number of times the summand $\Lambda/\Phi_n$ appearing on the right-hand side of the pseudo-isomorphism given by Theorem~\ref{thm:structure}. 
    \item  As in the introduction, we define $e_n$ to be $$\frac{\rank E(\QQ_{(n)})-\rank E(\QQ_{(n-1)})}{\phi(p^n)},$$
 where $E(\QQ_{(-1)})$ is taken to be the trivial group by convention.
\end{itemize} 
\end{defn}

By analyzing the Galois action on $E(\QQ_{(n)})$, we have the following isomorphism of $\Lambda$-modules
\begin{align*}
 E(\QQ_{(n)})^\star&\cong\bigoplus_{m=0}^n(\Zp[X]/\Phi_m)^{e_m},\\
 E(\QQ_{(n)})^\bullet&\cong\bigoplus_{m=0}^n(\Qp[X]/\Phi_m)^{e_m}.   
\end{align*}

\begin{lemma}\label{lem:surj-Phi_n}
Suppose that $e_n>0$ and let $f_n:E(\QQ_{(n)})^\bullet\rightarrow E(\QQ_{n,p})^\bullet $ denote the natural map induced by the inclusion $E(\QQ_{(n)})\hookrightarrow E(\QQ_{n,p})$. Then there is a $\Lambda$-isomorphism $$(\Image f_n)[\Phi_n]\cong \Qp[X]/\Phi_n.$$
\end{lemma}
\begin{proof}
For $n\ge0$, let $G_n=\Gal(\QQ_{n,p}/\Qp)$. By standard results on the structure of elliptic curves over local fields, we have a short exact sequence of  $G_n$-modules
\[
0\rightarrow \hat E(\mathfrak{m}_{n,p})\rightarrow E(\QQ_{n,p})\rightarrow (\text{finite group})\rightarrow0,
\]
where $\hat E$ denotes the formal group attached to $E$ at $p$ and $\mathfrak{m}_{n,p}$ is the maximal ideal of the ring of integers of $\QQ_{n,p}$. On tensoring by $\Qp$, the logarithm on $\hat E$ gives the isomorphisms of $\Qp[G_n]$-modules
\[
E(\QQ_{n,p})^\bullet\cong \QQ_{n,p}\cong \Qp[G_n].
\]
Hence, we have the following isomorphisms of $\Lambda$-modules:
\begin{align*}
&E(\QQ_{n,p})^\bullet\cong E(\QQ_{n-1,p})^\bullet\oplus\Qp[X]/\Phi_n    ,\\
&E(\QQ_{n,p})^\bullet[\Phi_n]\cong\Qp[X]/\Phi_n.
\end{align*}
 Furthermore, as $\Qp[X]/\Phi_n$ is an irreducible $G_n$-representation, it is enough to show that $(\Image f_n)[\Phi_n]$ is non-zero. Since $e_n>0$, there exists a non-torsion point $P\in E(\QQ_{(n)})\setminus E(\QQ_{(n-1)})$. In particular, it is a non-torsion point in $E(\QQ_{n,p})\setminus E(\QQ_{n-1,p})$. Therefore, it gives a non-zero element in $(\Image f_n)[\Phi_n]$ as required.
\end{proof}

\begin{corollary}\label{cor:sn-en}
We have $s_n=\max(0,e_n-1)$.
\end{corollary}
\begin{proof}

Consider the tautological exact sequence
\[
0\rightarrow M_{p^k}(E/\QQ_{(n)})\rightarrow E(\QQ_{(n)})/p^k\rightarrow E(\QQ_{n,p})/p^k.
\]
On taking inverse limits as $k$ varies and tensoring by $\Qp$, we deduce the following exact sequence:
\[
0\rightarrow V_p\cM(E/\QQ_{(n)})\rightarrow E(\QQ_{(n)})^\bullet\stackrel{f_n}{\rightarrow} E(\QQ_{n,p})^\bullet,
\]
where $f_n$ is given as in Lemma~\ref{lem:surj-Phi_n}.

On the one hand, Corollary~\ref{cor:Tate-fine} says that
\[
V_p\cM(E/\QQ_{(n)})\cong \bigoplus_{m=0}^n(\Qp[X]/\Phi_m)^{s_m}.
\]
On the other hand, we have the following isomorphism of $\Lambda$-modules
\[
E(\QQ_{(n)})^\bullet\cong \bigoplus_{m=0}^n(\Qp[X]/\Phi_m)^{e_m}.
\]
Therefore, on taking the kernels of $\Phi_n$, we obtain the exact sequence of $\Lambda$-modules
\[
0\rightarrow (\Qp[X]/\Phi_n)^{s_n}\rightarrow (\Qp[X]/\Phi_n)^{e_n}\rightarrow\Qp[X]/\Phi_n.
\]
Clearly, if $e_n=0$, we must have $s_n=0$. Otherwise, the last map is a surjection by Lemma~\ref{lem:surj-Phi_n}, forcing $s_n=e_n-1$. Hence the result follows.
\end{proof}

\begin{remark}
Let $r_n$ denote the number of times the direct summand  $\Lambda/\Phi_n$ appears in the pseudo-isomorphism \eqref{eq:pseudo-Lee}. 
If $E$ has good ordinary reduction at $p$, then  $e_n=r_n$ by \cite[Remark~2.1.6(2)]{lee20}. Therefore, we have $s_n=\max(0,r_n-1)$. If the fine Mordell--Weil groups do not satisfy a control theorem, it is not clear to us how one may relate $e_n$ to $r_n$ for an arbitrary $n\ge0$.
 \end{remark}

We can now prove Theorem~\ref{thmC}:

\begin{theorem}\label{thm:char-fineMW}
Suppose that $\sha(E/\QQ_{(n)})[p^\infty]$ is finite for all $n\ge0$. Then,
\[
\Char_\Lambda \cM(E/\Qcyc)^\vee=\left(\prod_{e_n>0}\Phi_n^{e_n-1}\right).
\]
\end{theorem}
\begin{proof}
By the definition of $s_n$, we have
\[
\Char_\Lambda\cM(E/\Qcyc)^\vee=\left(\prod_{n\ge0}\Phi_n^{s_n}\right).
\]
Hence, the theorem follows from Corollary~\ref{cor:sn-en}.
\end{proof}

\section{Definition of plus and minus Mordell--Weil groups and basic properties}
We now suppose that $E/\QQ$ has good supersingular reduction at $p$ with $a_p(E)=0$. We continue to employ the notation ($\QQ_{(n)}$, $\Qcyc$, $\QQ_{n,p}$, etc) introduced in the previous section.
\begin{defn}
 We define
 \[
 \cM^\pm(E/\Qcycp)=\bigcup_{n\ge0}E^\pm(\QQ_{n,p})\otimes\Qp/\Zp,
 \]
 where $E^\pm(\QQ_{n,p})$ are Kobayashi's plus and minus subgroups defined in \cite[Definition~1.1]{kobayashi03}. 
\end{defn}

We recall that via the Kummer map, $\cM^\pm(E/\Qcycp)$ can be considered as subgroups of $H^1(\QQ_{\cyc,p}, \Ep)$. Furthermore, since $E(\Qcycp)[p^\infty]=0$ (see \cite[Proposition~8.7]{kobayashi03}), there is an isomorphism
\[
H^1(\Qcycp,\Ep)^{\Gamma_n}\cong H^1(\QQ_{n,p},\Ep)
\]
given by the inflation-restriction exact sequence. Similarly, the exact sequence
\[
0\rightarrow E[p^k]\rightarrow \Ep\stackrel{\times p^k}\longrightarrow \Ep\rightarrow 0
\]
gives the isomorphism
\[
H^1(\QQ_{n,p},E[p^k])\cong H^1(\QQ_{n,p},\Ep)[p^k].
\]
These observations enable us to define the following:

\begin{defn}\label{def:pm}
We define for $n\ge0$
 \[
 \cM^\pm(E/\QQ_{n,p}):=\cM^\pm(E/\QQ_{\cyc,p})^{\Gal(\Qcycp/\QQ_{n,p})}\subset H^1(\QQ_{n,p},\Ep).
 \]
If  $k\ge1$ is an integer, we further define
 \[
 \cM_{p^k}^\pm(E/\QQ_{n,p}):=\cM^\pm(E/\QQ_{n,p})[p^k]\subset H^1(\QQ_{n,p},E[p^k]).
 \]

 Similar to \cite{wut-camb}, we define the $p^k$- and $p$-primary plus and minus Mordell--Weil groups of $E$ over $\QQ_{(n)}$ by
 \begin{align*}
       \cM_{p^k}^\pm(E/\QQ_{(n)})&:=\ker\left(E(\QQ_{(n)})/p^k\rightarrow \frac{H^1(\QQ_{n,p},E[p^k])}{\cM_{p^k}^\pm(E/\QQ_{n,p})}\right),\\
      \cM^\pm(E/\QQ_{(n)})&:=\varinjlim_k\cM_{p^k}^\pm(E/\QQ_{(n)})=\ker\left(E(\QQ_{(n)})\otimes\Qp/\Zp\rightarrow
      \frac{H^1(\QQ_{n,p},E[p^\infty])}{\cM^\pm(E/\QQ_{n,p})}\right).
 \end{align*}

 We  define the ($p$-primary) plus and minus Selmer groups   of $E$ over $\QQ_{(n)}$, denoted by $\Sel_{p^\infty}^\pm(E/\QQ_{(n)})$, to be the kernel
 \[\ker\left(H^1(\QQ_{(n)},\Ep)\rightarrow \frac{H^1(\QQ_{n,p},E[p^\infty])}{\cM^\pm(E/\QQ_{n,p})}\times\prod_{v\nmid p}H^1(\QQ_{(n),v},\Ep)\right).
 \]

 We define similarly $\cM^\pm(E/\Qcyc)$ and $\Sel_{p^\infty}(E/\Qcyc)$ on replacing the local conditions $\cM^\pm(E/\QQ_{n,p})$ by $\cM^\pm(E/\Qcycp)$.
 \end{defn}
 After identifying $E(\QQ_{(n)})\otimes\Qp/\Zp$ with a subgroup of $H^1(\QQ_{(n)},\Ep)$ via the Kummer map, we see that $\cM^\pm(E/\QQ_{(n)})$ is a subgroup of $\Sel^\pm(E/\QQ_{(n)})$. This allows us to give the following definition:
 \begin{defn}
For $K=\QQ_{(n)}$ or $\Qcyc$, we define the ($p$-primary) plus and minus  Tate--Shafarevich groups $\zhe^\pm(E/K)$ of $E$ over $K$ to be the quotients $\Sel_{p^\infty}^\pm(E/K)/\cM^\pm(E/K)$.
\end{defn}
As before, we omit the terminology $p$-primary from our discussion.
\begin{remark}\label{rk:pm}We make a number of remarks on the plus and minus objects that we have just defined.
\begin{itemize}
    \item We recall from \cite[Propositions~2.2 and 2.3]{kimdoublysigned} that $\cM^\pm(E/\QQ_{\cyc,p})^\vee$ is free of rank one over $\Lambda$, so $\cM^\pm(E/\QQ_{n,p})^\vee\cong \Zp[X]/\omega_n$, where $\omega_n=(1+X)^{p^n}-1$.
    \item  Analogous to \eqref{eq:fineSel}, we have the tautological short exact sequence:
\begin{equation}
    0\rightarrow \cM^\pm(E/K)\rightarrow \Sel_{p^\infty}^\pm(E/K)\rightarrow \zhe^\pm(E/K)\rightarrow0
\label{eq:SES-pm}
\end{equation}
for $K=\QQ_{(n)}$ or $\Qcyc$.
\item The plus and minus Selmer groups over $\QQ_{(n)}$ that we have  defined are not the same as the ones considered in \cite{kobayashi03} since $E^\pm(\QQ_{n,p})\otimes\Qp/\Zp$ is a proper subgroup of $\cM^\pm(E/\QQ_{n,p})$ unless $n=0$ or $\infty$.
\item Similar plus and minus Mordell--Weil groups as well as plus and minus Tate--Shafarevich groups were also studied in \cite{HLV} in the anticyclotomic setting (see in particular Definitions~4.1 and~5.11 in op. cit.).
\end{itemize}

\end{remark}

We have the following analogues of Theorems~\ref{thm:structure} and \ref{thm:control}:
\begin{theorem} \label{thm:structure-pm} Let $m_n^\pm:\cM^\pm(E/\QQ_{(n)})\rightarrow \cM^\pm(E/\Qcyc)^{\Gamma_n}$ be the natural morphism. 
\begin{itemize}
    \item[(a)]We have a pseudo-isomorphism
    \[
    \cM^\pm(E/\Qcyc)^\vee\sim \bigoplus_{i=1}^{u^\pm}\Lambda/\Phi_{c^\pm_i}
    \]
    for certain non-negative integers $u^\pm$ and $c_i^\pm$.
    \item[(b)] The kernel of $m_n^\pm$ is $0$ for all $n$. 
    \item[(c)]Suppose that $\zhe^\pm(E/\QQ_{(n)})$ is finite, then the cokernel of $m_n^\pm$ is finite.
\end{itemize}
\end{theorem}
\begin{proof}
Part (a) follows from the same proof as Theorem~\ref{thm:structure} since $ \cM^\pm(E/\Qcyc)^\vee$ is a quotient of $(E(\Qcyc)\otimes\Qp/\Zp)^\vee$, which is torsion over $\Lambda$ as discussed in Remark~\ref{rk:Kato-Roh}.

The proofs for parts (b) and (c) are similar to Theorem~\ref{thm:control}. Consider the commutative diagram
\[
\xymatrix{
0\ar[r]&\cM^\pm(E/\Qcyc)^{\Gamma_n}\ar[r]&\Sel_{p^\infty}^\pm(E/\Qcyc)^{\Gamma_n}\ar[r]&\zhe^\pm(E/\Qcyc)^{\Gamma_n}\\
0\ar[r]&\cM^\pm(E/\QQ_{(n)})\ar[u]^{m_n^\pm}\ar[r]&\Sel_{p^\infty}^\pm(E/\QQ_{(n)})\ar[u]^{\beta^\pm_n}\ar[r]&\zhe^\pm(E/\QQ_{(n)})\ar[u]\ar[r]&0.
}
\]
The map $\beta^\pm_n$ is injective since  $H^1(\QQ_{(n)},\Ep)\rightarrow H^1(\Qcyc,\Ep)$ is injective, which is a consequence of the inflation-restriction exact sequence and the fact that $H^0(\Qcyc,\Ep)=0$. This proves part (b). By \cite[Lemma~2.3]{ponsinet}, the cokernel of $\beta_n$ is finite. Therefore, part (c) follows from the snake lemma.
\end{proof}

\begin{remark}\label{rk:Tp-pm}
The proof of Corollary~\ref{cor:Tate-fine} can be applied to $\cM^\pm(E/\QQ_{(n)})$, telling us that
\[
T_p\cM^\pm(E/\QQ_{(n)})=\bigoplus_{c_i^\pm\le n}\Lambda/\Phi_{c_i^\pm}.
\]
\end{remark}

Similar to Remark~\ref{rk:finite}, $\zhe^\pm(E/\QQ_{(n)})$ is finite if  $\sha(E/K)[p^\infty]$ is finite, thanks to the following lemma.
\begin{lemma}
Let $K=\QQ_{(n)}$ or $\Qcyc$. There is an inclusion $\zhe^\pm(E/K)\subset \sha(E/K)[p^\infty]$.
\end{lemma}
\begin{proof}
Our argument follows closely to the one presented in \cite[P.3]{wut-camb}. Let $\Sel_{p^\infty}(E/K)$ be the classical $p$-primary Selmer group of $E$ over $K$, namely 
\[
\ker\left(H^1(K,\Ep)\rightarrow\prod_{v}\frac{H^1(K_v,\Ep)}{E(K_v)\otimes\Qp/\Zp}\right).
\]
On comparing this with the definition of $\Sel^\pm_{p^\infty}(E/K)$, we see that
\[
\Sel_{p^\infty}^\pm(E/K)\bigcap\Sel_{p^\infty}(E/K)=\ker\left(\Sel_{p^\infty}(E/K)\rightarrow \frac{H^1(K_p,\Ep)}{\cM^\pm(E/K_p)}\right)
\]
since $E(K_v)\otimes\Qp/\Zp=0$ for $v\nmid p$.
Now consider the commutative diagram
\[
\xymatrix{
0\ar[r]&E(K)\otimes\Qp/\Zp\ar[r]\ar[d]&\Sel_{p^\infty}(E/K)\ar[r]\ar[d]&\sha(E/K)[p^\infty]\ar[r]\ar[d]&0\\
0\ar[r]&\frac{H^1(K_p,\Ep)}{\cM^\pm(E/K_p)}\ar[r]^=& \frac{H^1(K_p,\Ep)}{\cM^\pm(E/K_p)}\ar[r]&0.&
}
\]
The snake lemma gives the following exact sequence
\[
0\rightarrow \cM^\pm(E/K)\rightarrow \Sel_{p^\infty}^\pm(E/K)\bigcap\Sel_{p^\infty}(E/K)\rightarrow \sha(E/K)[p^\infty],
\]
from which the result follows.
\end{proof}

In our proof of Theorem~\ref{thmD}, we shall utilize the following "modulo $p^k$" control theorem:

\begin{proposition}\label{prop:control}
Let $n\ge0$ and $k\ge1$ be integers.
The natural map $E(\QQ_{(n)})/p^k\rightarrow E(\QQ_{(n)})\otimes\Qp/\Zp$ given by $P\mod p^k\mapsto P\otimes p^{-k}$ is injective, with image given by $\left(E(\QQ_{(n)})\otimes\Qp/\Zp\right)[p^k]$. Furthermore, it induces the following isomorphisms of $\Lambda$-modules:
\begin{align*}
    \cM_{p^k}^\pm(E/\QQ_{(n)})&\cong \cM^\pm(E/\QQ_{(n)})[p^k],\\
      \cM_{p^k}(E/\QQ_{(n)})&\cong \cM(E/\QQ_{(n)})[p^k].
\end{align*}
\end{proposition}
\begin{proof}
The first assertion of the proposition is a consequence of $E(\QQ_{(n)})[p^\infty]=0$. For the second assertion, we consider the plus and minus Mordell--Weil groups (the proof for the fine Mordell--Weil group is the same). By definitions, we have an injection.
\[
\frac{H^1(\QQ_{n,p},E[p^k])}{\cM_{p^k}^\pm(E/\QQ_{n,p})}\hookrightarrow\frac{H^1(\QQ_{n,p},\Ep)}{\cM^\pm(E/\QQ_{n,p})}.
\]
Therefore, the result follows once we apply the snake lemma to the following commutative diagram:
\[
\xymatrix{
0\ar[r]&\cM^\pm(E/\Qcyc)[p^k]\ar[r]&\left(E(\QQ_{(n)})\otimes\Qp/\Zp\right) [p^k]\ar[r]&\frac{H^1(\QQ_{n,p},\Ep)}{\cM^\pm(E/\QQ_{n,p})}\\
0\ar[r]&\cM_{p^k}^\pm(E/\QQ_{(n)})\ar[u]\ar[r]&E(\QQ_{(n)})/p^k\ar[u]^{\cong}\ar[r]&\frac{H^1(\QQ_{n,p},E[p^k])}{\cM_{p^k}^\pm(E/\QQ_{n,p})}\ar[u]\ar[r]&0.
}
\]
\end{proof}
On taking inverse limits as $k$ varies, we deduce:
\begin{corollary}\label{cor:tate}
We have the following isomorphisms of $\Lambda$-modules:
\begin{align*}
T_p\cM^\pm(E/\QQ_{(n)})&=\varprojlim_k \cM_{p^k}^\pm(E/\QQ_{(n)}),\\
T_p\cM(E/\QQ_{(n)})&=\varprojlim_k \cM_{p^k}(E/\QQ_{(n)}).
\end{align*}
\end{corollary}

\section{Proof of Theorem~\ref{thmD}}
To simplify our exposition, we give the following definition:
\begin{defn}
We define $r_n^\pm$ to be the number of factors $\Lambda/\Phi_n$ appearing in the pseudo-isomorphism given in Theorem~\ref{thm:structure-pm}(a).
\end{defn}

In order to prove Theorem~\ref{thmD}, we study certain explicit relations between $r_n^\pm$ and $e_n$. We do so by studying the intersection and the sum of the plus and minus Mordell--Weil groups. We will also have to study two separate cases, namely $n=0$ and $n>0$. We begin with the following preliminary lemmas, which will allow us to study the intersection.
\begin{lemma}\label{lem:local-pm}
We have
\[
\cM_{p^k}^+(E/\QQ_{n,p})\bigcap\cM_{p^k}^-(E/\QQ_{n,p})= E(\Qp)/p^k.
\]
\end{lemma}
\begin{proof}
It follows from \cite[(8.22)]{kobayashi03} that 
\[
\cM^+(E/\Qcycp)\bigcap\cM^-(E/\Qcycp)=E(\Qp)\otimes\Qp/\Zp.
\]
Since $\Gamma_n$ acts trivially on the right-hand side, taking $\Gamma_n$-invariants gives
\[
\cM^+(E/\QQ_{n,p})\bigcap\cM^-(E/\QQ_{n,p})=E(\Qp)\otimes\Qp/\Zp.
\]
Thanks to Proposition~\ref{prop:control}, the result follows on taking $p^k$-torsions.
\end{proof}

\begin{lemma}\label{lem:global-pm}\hfill\begin{itemize}
 \item[(a)]We have $\cM_{p^k}^\pm(E/\QQ)=E(\QQ)/p^k$.
\item[(b)]For $n>0$, $V_p\left(\cM^+(E/\QQ_{(n)})\bigcap\cM^-(E/\QQ_{(n)})\right)[\Phi_n]=V_p\cM(E/\QQ_{(n)})[\Phi_n]$.
\end{itemize}

\end{lemma}
\begin{proof}
To prove the assertion in part (a), it suffices to show that
\begin{equation}\label{eq:claimed}
   \cM_{p^k}^\pm(E/\Qp)=E(\Qp)/p^k, 
\end{equation}
since it would imply that
\[
\cM_{p^k}^\pm(E/\QQ)=\ker\left(E(\QQ)/p^k\rightarrow\frac{H^1(\Qp,E[p^k])}{E(\Qp)/p^k}\right),
\]
which is clearly $E(\QQ)/p^k$.

Indeed, it follows immediately from the definition of $E^\pm(\Qcycp)$ that $E(\Qp)\otimes\Qp/\Zp\subset (E(\Qp)\otimes\Qp/\Zp)^{\Gamma}= \cM^\pm(E/\Qp)$. Since $E(\Qp)$ is a torsion-free rank-one $\Zp$-module, we have $E(\Qp)\otimes\Qp/\Zp\cong\Qp/\Zp$. The freeness of $\cM^\pm(E/\Qcycp)^\vee$ as discussed in Remark~\ref{rk:pm} tells us that $\cM^\pm(E/\Qp)\cong\Qp/\Zp$. In particular, $\cM(E/\Qp)=E(\Qp)\otimes\Qp/\Zp$. This implies \eqref{eq:claimed} after taking $p^k$-torsions and applying Proposition~\ref{prop:control}.

We now prove part (b) and let $n>0$. The description of the intersection of the local conditions given in Lemma~\ref{lem:local-pm} gives the following exact sequence:
\[
0\rightarrow V_p\left(\cM^+(E/\QQ_{(n)})\bigcap \cM^-(E/\QQ_{(n)})\right)\rightarrow E(\QQ_{(n)})^\bullet\rightarrow \frac{E(\QQ_{n,p})^\bullet}{E(\Qp)^\bullet}.
\]
Taking the kernels of $\Phi_n$ gives
\[
0\rightarrow V_p\left(\cM^+(E/\QQ_{(n)})\bigcap \cM^-(E/\QQ_{(n)})\right)[\Phi_n]\rightarrow E(\QQ_{(n)})^\bullet[\Phi_n]\rightarrow E(\QQ_{n,p})^\bullet[\Phi_n].
\]
The proof of Corollary~\ref{cor:sn-en} then tells us that 
\[
V_p\left(\cM^+(E/\QQ_{(n)})\bigcap \cM^-(E/\QQ_{(n)})\right)[\Phi_n]=(\Qp[X]/\Phi_n)^{\max(e_n-1,0)}
\]
 Hence, part (b) of the lemma now follows from Corollary~\ref{cor:sn-en}.
\end{proof}

\begin{corollary}\label{cor:rn-pm}\hfill\begin{itemize}
   \item[(a)]If $n=0$, then $r_n^+(E)=r_n^-(E)=e_n$.
\item[(b)]If $n>0$, then  $s_n=\max(0,e_n-1)\le \min(r_n^+,r_n^-)$.
\end{itemize}
\end{corollary}
\begin{proof}
By Remark~\ref{rk:Tp-pm},
\[
V_p\cM^\pm(E/\QQ_{(n)})=\bigoplus_{m=0}^n(\Qp[X]/\Phi_m)^{r_m^\pm}.
\]
If $n=0$, we have $$V_p\cM^\pm(E/\QQ)=\Qp^{e_0}$$ on combining Corollary~\ref{cor:tate} and  Lemma~\ref{lem:global-pm}(a). Therefore, part (a) of the corollary follows.

For $n>0$,  we have 
\[
V_p\cM(E/\QQ_{(n)})[\Phi_n]\subset V_p\cM^\pm(E/\QQ_{(n)})[\Phi_n]
\]
by Lemma~\ref{lem:global-pm}(b). Therefore, part (b) of the corollary  follows from  Corollaries~\ref{cor:sn-en} and \ref{cor:Tate-fine}.
\end{proof}

We now turn our attention to the sum of plus and minus Mordell--Weil groups.

\begin{lemma}\label{lem:sum-pm}
The quotient
\[
\frac{E(\QQ_{(n)})/p^k}{\cM_{p^k}^+(E/\QQ_{(n)})+\cM_{p^k}^-(E/\QQ_{(n)})}\]
is finite, of order bounded independently of $k$.
\end{lemma}
\begin{proof}We mimic  the proof of \cite[Proposition~10.1]{kobayashi03}.  Let 
\[
\tilde\omega_n^-=\prod_{\substack{1\le m\le n\\ m \text{ odd}}}\Phi_m,\quad
    \omega_n^+=\prod_{\substack{0\le m\le n\\ m \text{ even}}}\Phi_m.
\]
Since these polynomials are coprime,  there exist $A,B\in\Zp[X]$ and an integer $m\ge0$ such that
\[
A\tilde\omega^-_n+B\omega_n^+=p^m.
\]
If $Q\in E(\QQ_{(n)})/p^k$, we may write $$p^mQ=P^++P^-,$$
where $P^+=A\tilde\omega^-_nQ$ and $P^-=B\tilde\omega^+_nQ$. Then, 
\[
\loc_pP^\pm\in E^\pm(E/\QQ_{n,p})/p^k\subset \cM_{p^k}^\pm(E/\QQ_{n,p}),
\]
where $\loc_p$ denotes the localization map at the unique place of $\QQ_{(n)}$ lying above $p$.
Therefore, $P^\pm\in\cM^\pm_{p^k}(E/\QQ_{(n)})$. In particular,  the quotient discussed in the statement of  the lemma is bounded by $p^m\times\rank E(\QQ_{(n)})$.
\end{proof}

\begin{remark}
We have crucially made use of the hypothesis that $a_p(E)=0$ in the proof of Lemma~\ref{lem:sum-pm} above. When $a_p(E)\ne0$, we may define sharp/flat Mordell--Weil groups using Sprung's local conditions studied in \cite{sprung09}. However, in the absence of an explicit description of these local conditions, it is not clear to us how to prove an analogue of Lemma~\ref{lem:sum-pm} in this setting.
\end{remark}

\begin{corollary}\label{cor:sum}
If $n>0$, then 
\[ 
r^+_n+r_n^-=e_n+s_n.
\]
\end{corollary}
\begin{proof}
Consider the exact sequence
\begin{align*}
    0\rightarrow  \cM_{p^k}^+(E/\QQ_{(n)})\bigcap \cM_{p^k}^-(E/\QQ_{(n)})&\rightarrow \cM_{p^k}^+(E/\QQ_{(n)})\oplus  \cM_{p^k}^-(E/\QQ_{(n)})\\
    &\rightarrow E(\QQ_{(n)})/p^k,
\end{align*}
where the first map is given by the diagonal embedding, whereas the second map is defined by $P\oplus Q\mapsto P-Q$. On taking  inverse limits  and tensoring by $\Qp$, we obtain the short exact sequence
\begin{align*}
    0\rightarrow  V_p\left(\cM^+(E/\QQ_{(n)})\bigcap \cM^-(E/\QQ_{(n)})\right)&\rightarrow V_p \cM^+(E/\QQ_{(n)})\oplus V_p \cM^-(E/\QQ_{(n)})\\
    &\rightarrow E(\QQ_{(n)})^\bullet\rightarrow 0,
\end{align*}
where the surjectivity of the last map follows from Lemma~\ref{lem:sum-pm}.
On combining this short exact sequence with Corollary~\ref{cor:Tate-fine}, Remark~\ref{rk:Tp-pm} and Lemma~\ref{lem:global-pm}(b), we deduce the following short exact sequence:
\[
0\rightarrow (\Qp[X]/\Phi_n)^{s_n}\rightarrow(\Qp[X]/\Phi_n)^{r_n^+}\oplus (\Qp[X]/\Phi_n)^{r_n^-}\rightarrow (\Qp[X]/\Phi_n)^{e_n}\rightarrow0,
\]
from which the corollary follows.
\end{proof}
We can now prove the following key proposition, which will allow us to calculate the left-hand side of the equation in the statement of Theorem~\ref{thmD}:
\begin{proposition}\label{prop:rn-pm}
Let $n>0$. We have the equalities
\[
\min(r_n^+,r_n^-)=s_n=\max(0,e_n-1).
\]
\end{proposition}
\begin{proof}
If $e_n=0$, it is clear that $r_n^+=r_n^-=0$  since $\cM^\pm(E/\QQ_{(n)})\subset E(\QQ_{(n)})\otimes\Qp/\Zp$. 

Suppose that $e_n>0$. It follows from Corollaries~\ref{cor:sum} and \ref{cor:sn-en} that
\[
r_n^++r_n^-=2e_n-1.
\]
In particular, $r_n^+\ne r_n^-$ (otherwise, the left-hand side would be an even integer, whereas the right-hand side is always odd). Without loss of generality, assume that $r_n^-+1\le r_n^+$. Then
$$
2r_n^-+1\le 2e_n-1,
$$
or equivalently,
$$
r_n^-\le e_n-1,
$$
But Corollary~\ref{cor:rn-pm}(b) says that $$ e_n-1=s_n\le r_n^-.$$
Therefore, we must have equalities throughout, meaning that  $e_n-1=s_n=r_n^-$ and $r_n^-+1=r_n^+$ as required.
\end{proof}
We can now prove Theorem~\ref{thmD}:

\begin{theorem}\label{thm:Char-pm}
Suppose that $\zhe^\pm(E/\QQ_{(n)})$ are finite for all $n\ge0$. Then
\[
\gcd\left(\Char_\Lambda\cM^+(E/\Qcyc)^\vee,\Char_\Lambda\cM^-(E/\Qcyc)^\vee\right)=\left(X^{e_0}\prod_{n>0,e_n>0}\Phi_n^{e_n-1}\right).
\]
\end{theorem}
\begin{proof}
By definition,
\[
\Char_\Lambda\cM^\pm(E/\Qcyc)^\vee=\left(\prod_{n\ge0}\Phi_n^{r_n^\pm}\right).
\]
In particular,
\[
\gcd\left(\Char_\Lambda\cM^+(E/\Qcyc)^\vee,\Char_\Lambda\cM^-(E/\Qcyc)^\vee\right)=\left(\prod_{n\ge0}\Phi_n^{\min(r_n^+,r_n^-)}\right).
\]
Corollary~\ref{cor:rn-pm} says that $r_n^\pm=e_n$ when $n=0$, whereas Proposition~\ref{prop:rn-pm} tells us that $\min(r_n^+,r_n^-)=s_n=\max(0,e_n-1)$ for $n>0$. Therefore, the theorem follows.
\end{proof}
\begin{remark}\label{rk:KP}
Suppose that $\zhe^\pm(E/\Qcyc)$ are both finite and that Kobayashi's plus and minus main conjectures  hold (see \cite[Page 2]{kobayashi03}). Then, it follows from \eqref{eq:SES-pm} that
\[
\Char_\Lambda\cM^\pm(E/\Qcyc)^\vee=\Char_\Lambda\Sel_{p^\infty}^\pm(E/\Qcyc)^\vee=(L_p^\pm).
\]
In particular, Theorem~\ref{thmD} tells us that
\[
\gcd(L_p^+,L_p^-)=\left(X^{e_0}\prod_{n>0,e_n>0}\Phi_n^{e_n-1}\right)
\]
as predicted by Kurihara--Pollack's problem.
\end{remark}

\bibliographystyle{amsalpha}
\bibliography{references}

\end{document}